\documentclass[10pt,leqno]{amsart}
\usepackage{graphicx}
\usepackage{indentfirst,csquotes}
\usepackage[indentafter]{titlesec}
\titleformat{name=\section}{}{\thetitle.}{0.8em}{\centering\scshape}
\titleformat{name=\subsection}[runin]{}{\thetitle.}{0.5em}{\bfseries}[.]
\titleformat{name=\subsubsection}[runin]{}{\thetitle.}{0.5em}{\itshape}[.]
\titleformat{name=\paragraph,numberless}[runin]{}{}{0em}{}[.]
\titlespacing{\paragraph}{0em}{0em}{0.5em}
\titleformat{name=\subparagraph,numberless}[runin]{}{}{0em}{}[.]
\titlespacing{\subparagraph}{0em}{0em}{0.5em}
\usepackage{lipsum}

\usepackage{amssymb,amsthm,amsmath}
\usepackage{xcolor,paralist,hyperref,titlesec,fancyhdr,etoolbox,diagbox,multirow,mathtools,bm}
\usepackage[all]{xy}
\theoremstyle{definition}
\newtheorem{definition}{Definition}[]
\newtheorem{remark}[definition]{Remark}
\newtheorem{example}[definition]{Example}
\theoremstyle{plain}
\newtheorem{theorem}[definition]{Theorem}
\newtheorem{corollary}[definition]{Corollary}

\newtheorem{lemma}[definition]{Lemma}
\newtheorem{proposition}[definition]{Proposition}
\renewenvironment{proof}{\noindent\textsc{Proof.}\quad}{\qed}

\hypersetup{ colorlinks=true, linkcolor=black, filecolor=black, urlcolor=black }

\newcommand{\R}{\mathbb{R}}
\newcommand{\Q}{\mathbb{Q}}
\newcommand{\Z}{\mathbb{Z}}

\newcommand{\E}{\mathbb{E}}
\newcommand{\Prob}{\mathbb{P}}
\newcommand{\Var}{\operatorname{Var}}

\newcommand{\im}{\operatorname{im}}
\newcommand{\Spec}{\operatorname{Spec}}
\newcommand{\mc}{\mathrm{mc}}
\newcommand{\sh}{\mathrm{sh}}
\newcommand{\can}{\mathrm{can}}
\newcommand{\one}{\mathbf{1}}
\newcommand{\TV}{\mathrm{TV}}
\begin{document}

\title[Canonical and Microcanonical Ensembles]{Equivalence of Canonical and Microcanonical Ensembles for Euclidean Lattices}
\author[Ziyang Zhu]{Ziyang Zhu}
\date{July 21, 2026}
\address{School of Mathematical Sciences, Capital Normal University, Beijing 100048, China}
\email{zhuziyang@cnu.edu.cn}

\begin{abstract}
We study the equivalence of canonical and microcanonical ensembles for quadratic Hamiltonians on Euclidean lattices. The answer depends on the arithmetic structure of the energy spectrum: exact-energy microcanonical ensembles converge locally to the canonical Gibbs law in the lattice case, but may retain additional arithmetic information in the nonlattice case. We show that this obstruction disappears after replacing an exact energy shell by any fixed-width energy window. Using uniform local central limit theorems, we also obtain sharp asymptotics for the corresponding state counts and recover Bost's thermodynamic entropy.
\end{abstract}

\maketitle

\let\thefootnote\relax

\section{Introduction}
\newtheorem*{maintheorem}{Theorem}

The canonical and microcanonical ensembles are two fundamental descriptions of equilibrium statistical mechanics. The canonical ensemble fixes the temperature and allows the energy to fluctuate, whereas the microcanonical ensemble constrains the total energy, either exactly or within a narrow energy window. At the level of fixed subsystems, ensemble equivalence expresses the principle that the microcanonical law of a small subsystem of a large isolated system should approach the corresponding canonical Gibbs law in the thermodynamic limit, with the inverse temperature chosen to match the prescribed specific energy. We refer to \cite{Ellis} for background on equilibrium statistical mechanics.

In this paper, we formulate this principle as local convergence of probability measures. More precisely, two sequences $\{\mu_N\}$ and $\{\nu_N\}$ of measures on $N$ copies of a phase space are called equivalent if, for every fixed $k\geq1$,
\[\|(\pi_k)_*\mu_N-(\pi_k)_*\nu_N\|_{\TV}\to0,\quad N\to+\infty,\]
where $(\pi_k)_*$ denotes projection onto the first $k$ coordinates and '$\TV$' means total variation. This subsystem interpretation of ensemble equivalence stems from an analogy with a heat bath, see, for example, \cite{Plastino}.

Let $(\Lambda,\|\cdot\|)$ be a Euclidean lattice and consider the quadratic kinetic-energy Hamiltonian $H(x)=\|x\|^2$ on the phase space $\Lambda$. Let $\beta>0$ be interpreted as the inverse temperature (Boltzmann's constant is set to $1$ for simplicity). Consider the following one-particle canonical measure on $\Lambda$,
\[p_{\beta}(x):=Z(\beta)^{-1}e^{-\beta H(x)},\quad Z(\beta):=\sum_{x\in\Lambda}e^{-\beta H(x)}.\]
There are some standard thermodynamic functions
\[\psi(\beta):=\log Z(\beta),\quad m(\beta):=\E_{p_{\beta}}[H],\quad\sigma^2(\beta):=\Var_{p_{\beta}}[H],\quad s(u):=\inf_{\theta>0}\{\psi(\theta)+\theta u\}.\]

For $N$ labelled, noninteracting copies with phase space $\Lambda^N$, the total Hamiltonian is
\[H_N(x_1,\cdots,x_N):=\sum_{i=1}^NH(x_i).\]
No quotient by particle permutations is taken; the sites are independent under the canonical product measure. The canonical measure on $\Lambda^N$ is defined as $\mu_{N,\beta}^{\can}:=p_{\beta}^{\otimes N}$, which is the entropy-maximizing law with kinetic energy $Nu$, where $m(\beta)=u$, see Proposition~\ref{prop:max-entropy}. However, there are two ways to define the microcanonical measure: the equal-a-priori measure $\mu_{N,E}^{\sh}$ on a reachable exact energy shell $H_N^{-1}(E)$, or the equal-a-priori measure $\mu_{N,E,\Delta}^{\mc}$ on a $\Delta$-fixed-width energy window $H_N^{-1}([E,E+\Delta))$.

The motivation comes from Bost's thermodynamic formalism \cite{Bost}. Although his measure-theoretic results allow a general nonnegative Hamiltonian, the examples most closely related to the present work are quadratic kinetic-energy models: Maxwell's model in the continuous setting and the Euclidean-lattice model with $H_{\mathrm{B}}(x)=\pi\|x\|^2$. Bost identified, through cumulative lattice-point counts in large direct sums, the Legendre-Fenchel duality between the logarithmic theta series and a stable Arakelov entropy \cite[Theorem 5.2.1]{Bost}. We ask whether this exponential-scale result can be refined to individual energy levels or fixed-width energy windows, and whether the resulting microcanonical measures are equivalent to the corresponding Gibbs law.

Our primary tools are local central limit theorems (LCLT). In \cite{Cancrini}, Cancrini and Olla established the local central limit theorems and local large deviations to prove the explicit formula for the variance differences of extensive observables between canonical and microcanonical ensembles, clarifying their equivalence and finite-size corrections. Furthermore, Diaconis and Freedman \cite{Diaconis} applied conditional limit theorems to study conditional laws of finitely many coordinates given a sum and their approximation by exponential-family product measures. In contrast, in applying the local central limit theorem, our focus is on the arithmetic properties of $H(\Lambda)$.

\begin{table}[h]
  \centering
\begin{tabular}{c|c|c}
  \hline
  & \textsf{Bost's Arguments} \cite{Bost} & \textsf{Our Arguments} \\
  \hline
  \textsf{Model} & Cumulative Spherical Shell & Exact/Windowed Energy Shell \\
  \hline
  \textsf{Method} & Large Deviation Theory & Local Central Limit Theorems \\
  \hline
  \textsf{Frame} & Classical Thermodynamics & Statistical Ensembles \\
  \hline
\end{tabular}
\end{table}

Our main result shows that whether canonical-microcanonical equivalence holds depends on the arithmetic structure of $H(\Lambda)$.

\begin{maintheorem}[Theorem~\ref{thm:h-equivalent}]
Suppose the abelian group $\langle H(x):x\in\Lambda\rangle_{\Z}$ has rank $r$ and fix a $\Z$-basis $\{a_1,\cdots,a_r\}$. Write
\[H(x)=a\cdot F(x)^T,\quad a=(a_1,\cdots,a_r),\quad F(x)\in\Z^r.\]
For $\lambda\in\R^r$ such that $\lambda\cdot F(x)^T$ is positive definite, define
\[\bm{p}_{\lambda}(x):=\frac{e^{-\lambda\cdot F(x)^T}}{\sum_{y\in\Lambda}e^{-\lambda\cdot F(y)^T}},\quad\bm{m}(\lambda):=\E_{\bm{p}_{\lambda}}[F].\]
Let $E_N=a\cdot n_N^T$ with $n_N\in\Z^r$, and assume the exact shell
\[\left\{(x_1,\cdots,x_N)\in\Lambda^N:\sum_{i=1}^NH(x_i)=E_N\right\}\]
is non-empty. If there exist $\lambda_N\to\lambda$ such that $n_N-N\bm{m}(\lambda_N)=O(1)$, then for every $k\geq1$,
\[\left\|(\pi_k)_*\mu^{\sh}_{N,E_N}-(\pi_k)_*\bm{p}_{\lambda}^{\otimes N}\right\|_{\TV}\to0,\quad N\to+\infty.\]
\end{maintheorem}

Our proof applies parameter-uniform local central limit theorems, which evaluate the relevant shell probabilities. The counting approach is inspired by the conditional limit theorem \cite{Diaconis}, and the local analysis is based on classical local limit theory: see \cite[Chapter 2]{Lawler} for the lattice case, Shepp and Stone \cite{Shepp, Stone} for nonlattice interval asymptotics.

As a corollary, if $H(\Lambda)$ is lattice, then $\bm{p}_{\lambda}$ degenerates to $p_{\beta}$, so the exact-shell microcanonical ensemble $\mu_{N,E}^{\sh}$ is equivalent to the canonical ensemble $\mu_{N,\beta}^{\can}$. However, when $H(\Lambda)$ is nonlattice, $\mu_{N,E}^{\sh}$ is generally not equivalent to $\mu_{N,\beta}^{\can}$. This breakdown is restored when considering windowed microcanonical ensembles $\mu_{N,E,\Delta}^{\mc}$, see Theorem~\ref{thm:equivalence}. In addition, we compute the entropy density $s(u)$ for both types of microcanonical ensembles. At exponential scale, lattice exact shells and nonlattice fixed-width windows recover Bost's entropy.  Thus our results give a local-scale refinement of Bost's cumulative lattice-point counting.

The paper is organized as follows. In \S\ref{sec:lattices-ensembles}, we introduce the canonical and microcanonical ensembles associated with quadratic Hamiltonians on Euclidean lattices and review the relevant thermodynamic functions. In \S\ref{sec:lattice-exact}, we study exact energy shells and prove the equivalence of canonical and microcanonical ensembles in the lattice case, based on a parameter-uniform local central limit theorem. In \S\ref{sec:window}, we turn to the nonlattice case and show that ensemble equivalence is recovered for microcanonical ensembles on fixed-width energy windows, using a uniform version of Stone's local central limit theorem. Finally, in \S\ref{sec:bost}, we compare our results with Bost's thermodynamic formalism and explain how the local state-counting asymptotics obtained here refine his cumulative lattice-point counting at the exponential scale.

\subsection*{Acknowledgements}
The author thanks capital normal university for opening its gymnasium during the summer holidays.

\section{Euclidean Lattices and the Canonical Ensembles}\label{sec:lattices-ensembles}
In this section, we begin with the discrete phase space model of single-particle kinetic energy to introduce the required ensemble theories. We will define the canonical ensemble in \S\ref{sec:canonical-ensemble} by mimicking the familiar Gibbs law in statistical mechanics. However, to construct the microcanonical ensemble, we employ two models: one considering the exact energy shell, and the other involving a thickened energy shell. These correspond to Definition~\ref{def:exact-shell} and Definition~\ref{def:window}, respectively. The choice of a more suitable model depends on the arithmetic structure of the Hamiltonian.
\subsection{The One-particle Model}\label{sec:one-model}
A Euclidean lattice is a pair $\overline\Lambda=(\Lambda,\|\cdot\|)$, where $\Lambda$ is a free $\Z$-module of rank $d\geq1$ and $\|\cdot\|$ is a Euclidean norm on $V=\Lambda\otimes_{\Z}\R$.  We identify $\Lambda$ with its image in $V$ and set $H(x)=\|x\|^2$.  Positive definiteness implies that every sublevel set $\{x\in\Lambda:H(x)\leq R\}$ is finite.

For $\beta>0$, define the partition function, logarithmic partition function, and one-particle Gibbs law by
\[Z(\beta):=\sum_{x\in\Lambda}e^{-\beta H(x)},\quad\psi(\beta):=\log Z(\beta),\quad p_{\beta}(x):=Z(\beta)^{-1}e^{-\beta H(x)}.\]
The series $Z(\beta)$ is the well-known theta series, which converges locally uniformly on $(0,+\infty)$.

Let $X_{\beta}$ be a random variable with law $p_{\beta}$ and let $Y_{\beta}=H(X_{\beta})$.  We write
\[m(\beta):=\E_{p_\beta}[Y_{\beta}],\quad\sigma^2(\beta):=\Var_{p_\beta}[Y_{\beta}].\]

\begin{proposition}\label{prop:thermodynamic-functions}
The functions $Z$, $\psi$, $m$, and $\sigma^2$ are real analytic on $(0,+\infty)$, and
\begin{equation}\label{eq:derivatives}
\psi'(\beta)=-m(\beta),\quad\psi''(\beta)=\sigma^2(\beta)>0,\quad m'(\beta)=-\sigma^2(\beta)<0.
\end{equation}
Moreover, $m(\beta)\to0$ as $\beta\to+\infty$ and $m(\beta)\to+\infty$ as $\beta\downarrow0$.  Hence $m:(0,+\infty)\to(0,+\infty)$ is a real-analytic decreasing bijection.
\end{proposition}

\begin{proof}
Local uniform convergence of all differentiated series justifies termwise differentiation and gives analyticity.  Formula \eqref{eq:derivatives} follows by direct calculation.  Since $p_{\beta}(x)>0$ for every $x\in\Lambda$ and $H$ is not constant, the variance is strictly positive.

As $\beta\to+\infty$, the numerator of $m(\beta)=Z(\beta)^{-1}\sum_xH(x)e^{-\beta H(x)}$ tends to zero by dominated convergence, while $Z(\beta)\geq1$ because $0\in\Lambda$.  Thus $m(\beta)\to0$.  As $\beta\downarrow0$, monotone convergence implies $Z(\beta)\to+\infty$.  For $R>0$, the finite set $B_R=\{H\leq R\}$ satisfies
\[\Prob_{p_\beta}(Y_{\beta}\leq R)\leq \#B_R/Z(\beta)\to0,\quad\beta\downarrow0.\]
Consequently $m(\beta)\geq R\Prob_{p_\beta}(Y_{\beta}>R)\to R$.  Letting $R\to+\infty$ proves the second limit.
\end{proof}

Physically, $m(\beta)$ represents the internal energy at inverse temperature $\beta$, which is precisely the average kinetic energy of the particles. Thus, Proposition~\ref{prop:thermodynamic-functions} simply asserts that the internal energy parametrizes the inverse temperature bijectively.

The Legendre-Fenchel transform of $\psi$ is defined as
\[s(u):=\inf_{\beta>0}\{\psi(\beta)+\beta u\},\quad u>0.\]
Since $\frac{d}{d\beta}(\psi(\beta)+\beta u)=-m(\beta)+u$, $s(u)$ admits a unique minimizer at $\beta=m^{-1}(u)$. The physical meaning of the Legendre-Fenchel transform is to convert the temperature-logarithmic partition function $\psi(\beta)$ into the internal energy-entropy function $s(u)$, where the definition of entropy will be given in \S\ref{sec:canonical-ensemble}.

The following proposition is immediate.

\begin{proposition}[Legendre-Fenchel Duality]\label{cor:duality}
The function $s$ is real analytic and strictly concave on $(0,+\infty)$, and
\[s'(u)=m^{-1}(u),\quad s''(u)=-\frac{1}{\sigma^2(m^{-1}(u))}<0.\]
Moreover,
\[\psi(\beta)=\sup_{u>0}\{s(u)-\beta u\}.\]
\end{proposition}

\subsection{The Many-site Canonical Ensemble}\label{sec:canonical-ensemble}
For $N\geq1$, let $H_N(x_1,\cdots,x_N)=\sum_{i=1}^NH(x_i)$ be the total Hamiltonian. Define the canonical ensemble on $\Lambda^N$ as
\[\mu^{\can}_{N,\beta}:=p_{\beta}^{\otimes N}.\]
Its internal energy and canonical (Shannon) entropy are
\[U_{N,\beta}:=Nm(\beta),\quad S(\mu_{N,\beta}^{\can}):=-\sum_{x\in\Lambda^N}\mu^{\can}_{N,\beta}(x)\log\mu^{\can}_{N,\beta}(x)=N(\psi(\beta)+\beta m(\beta))=Ns(m(\beta)).\]

The canonical ensemble primarily describes statistical mechanics in equilibrium states, namely, the states of maximum entropy.

\begin{proposition}[Gibbs Variational Principle]\label{prop:max-entropy}
Let $N\geq1$ and $\beta>0$.  Among all probability measures $\nu$ on $\Lambda^N$ satisfying $\E_{\nu}[H_N]=Nm(\beta)$, the canonical measure $\mu^{\can}_{N,\beta}$ uniquely maximizes the Shannon entropy
\[S(\nu):=-\sum_{x\in\Lambda^N}\nu(x)\log\nu(x).\]
\end{proposition}

\begin{proof}
We first justify finiteness of the entropy.  Let $\gamma>0$ and let $F\subseteq\Lambda^N$ be finite.  Applying the finite-state Gibbs inequality to the normalized restriction $\nu(\cdot|F)$ we obtain
\[-\sum_{x\in F}\nu(x)\log\nu(x)\leq-\nu(F)\log\nu(F)+\gamma\sum_{x\in F}\nu(x)H_N(x)+\nu(F)\log\sum_{x\in F}e^{-\gamma H_N(x)}.\]
Since $Z(\gamma)\geq1$, the right-hand side is at most
$-\nu(F)\log\nu(F)+\gamma\E_{\nu}[H_N]+N\psi(\gamma)$.  Hence, for finite sets $F$ increasing to $\Lambda^N$,
\[S(\nu)\leq\gamma\E_{\nu}[H_N]+N\psi(\gamma)<+\infty.\]
The cross-entropy with respect to $\mu^{\can}_{N,\beta}$ is finite as well, and the relative-entropy identity is therefore legitimate:
\[D(\nu\|\mu^{\can}_{N,\beta})=-S(\nu)+\beta\E_{\nu}[H_N]+N\psi(\beta)\geq0.\]
According to the assumption $\E_{\nu}[H_N]=Nm(\beta)$, we have $S(\nu)\leq S(\mu_{N,\beta}^{\can})$.  Equality holds if and only if the relative entropy vanishes, namely if and only if $\nu=\mu^{\can}_{N,\beta}$.
\end{proof}

\subsection{The Microcanonical Ensembles}
We adopt two approaches to define the microcanonical ensemble because the Hamiltonian is not always lattice. These definitions are analogous to the principle of equal-a-priori probabilities in statistical mechanics.
\begin{definition}[Lattice]
A subset $A\subseteq\R$ is lattice if $A\subseteq a+h\Z$ for some $a\in\R$ and $h>0$.  A real random variable is lattice-valued if its support is lattice.
\end{definition}
In the present setting $0\in H(\Lambda)$, so the energy set $H(\Lambda)$ is lattice if and only if it is contained in $h\Z$ for some $h>0$.
\subsubsection{The Exact-shell Microcanonical Ensemble}
This model is more appropriate when $H(\Lambda)$ is lattice. The reason for this will be discussed in \S\ref{sec:lattice-exact}.

\begin{definition}[Exact-shell Microcanonical Ensemble]\label{def:exact-shell}
For $E\in\im(H_N)$ (such an $E$ is referred to as reachable), denote
\[\Gamma_N(E):=\{x\in\Lambda^N:H_N(x)=E\},\quad g_N(E):=\#\Gamma_N(E).\]
Define the exact-shell microcanonical ensemble on $\Gamma_N(E)$ and its entropy by
\[\mu^{\sh}_{N,E}(x):=g_N(E)^{-1}\one_{\Gamma_N(E)}(x),\quad R^{\sh}_{N,E}:=\log g_N(E).\]
\end{definition}

It is clear that the exact-shell microcanonical ensemble is independent of $\beta$.

\begin{proposition}\label{prop:exact-conditioning}
For every $\beta>0$ and every reachable $E$,
\[\mu^{\sh}_{N,E}=\mu^{\can}_{N,\beta}(\cdot|H_N=E).\]
\end{proposition}
\begin{proof}
Every $x\in\Gamma_N(E)$ has canonical mass $Z(\beta)^{-N}e^{-\beta E}$.  Conditioning therefore assigns the common mass $g_N(E)^{-1}$ to every point of the shell.
\end{proof}

This proposition illustrates how to derive the (exact-shell) microcanonical ensemble from the canonical ensemble, consistent with the standard approach in classical thermodynamics.

\subsubsection{The Windowed Microcanonical Ensemble}
The model we shall give is primarily intended to handle nonlattice $H(\Lambda)$. Of course, it can also deal with lattice ones in practice.

\begin{definition}[Windowed Microcanonical Ensemble]\label{def:window}
Let $\Delta>0$. For $E\in\R$ and $N\geq1$, denote
\[\Omega_N(E,\Delta):=\{x\in\Lambda^N:E\leq H_N(x)<E+\Delta\},\quad G_N(E,\Delta):=\#\Omega_N(E,\Delta).\]
Whenever $G_N(E,\Delta)>0$, define the windowed microcanonical ensemble as the uniform measure
\[\mu^{\mc}_{N,E,\Delta}(x):=G_N(E,\Delta)^{-1}\one_{\Omega_N(E,\Delta)}(x)\]
on $\Omega_N(E,\Delta)$, and define its entropy as $R_{N,E,\Delta}^{\mc}:=\log G_N(E,\Delta)$.
\end{definition}

The set $\Omega_N(E,\Delta)$ is finite by positive definiteness.  The measure in Definition~\ref{def:window} is also independent of $\beta$.  By contrast, the conditional canonical measure on the same window assigns weights proportional to $e^{-\beta H_N(x)}$ and is therefore not uniform unless the window contains only one energy value.  This distinction prevents a circular definition of the microcanonical ensemble.

\section{Exact-shell Ensembles and Arithmetic Constraints}\label{sec:lattice-exact}
In this section, we study the thermodynamic limit behavior of the exact-shell microcanonical ensemble. In particular, if $H(\Lambda)$ is lattice, we establish the equivalence between the canonical and (exact-shell) microcanonical ensembles in the sense of total variation.

The images of $H$ generate an abelian group
\[G_H:=\langle H(x):x\in\Lambda\rangle_{\Z}\subseteq\R.\]
Since $H$ is a quadratic form on the finitely generated group $\Lambda$, the abelian group $G_H$ is torsion free and finitely generated. Fix a $\Z$-basis $\{a_1,\cdots,a_r\}$ of $G_H$, where $r=\mathrm{rank}_{\Z}(G_H)$. It is easy to see that $H(\Lambda)$ is lattice if and only if $r=1$.

The quadratic form $H$ uniquely induces a family of quadratic forms
\[F_1,\cdots,F_r:\Lambda\longrightarrow\Z\]
characterized by
\[H(x)=\sum_{j=1}^ra_jF_j(x).\]
Indeed, if $\{e_1,\cdots,e_d\}$ is a $\Z$-basis of $\Lambda$, we have
\[H\left(\sum_{i=1}^dn_ie_i\right)=\sum_{i=1}^dn_i^2H(e_i)+\sum_{i<j}n_in_j(H(e_i+e_j)-H(e_i)-H(e_j)),\]
expanding the coefficients in the basis $\{a_1,\cdots,a_r\}$ yields the quadratic forms $F_1,\cdots,F_r$, and the expansion is unique.

For $\lambda=(\lambda_1,\cdots,\lambda_r)\in\R^r$, define a quadratic form induced by $H$ and $\lambda$,
\[Q_{H,\lambda}:\Lambda\longrightarrow\R,\quad x\longmapsto\sum_{j=1}^r\lambda_jF_j(x),\]
consider
\[\mathcal{D}:=\left\{\lambda\in\R^r:Q_{H,\lambda}\text{ is positive definite}\right\}.\]
\begin{proposition}
Write $\mathcal{Z}(\lambda):=\sum_{x\in\Lambda}e^{-Q_{H,\lambda}(x)}$, then the set $\mathcal{D}$ is a non-empty open set, and it equals the interior of $\{\lambda\in\R^r:\mathcal{Z}(\lambda)<+\infty\}$.
\end{proposition}
\begin{proof}
$\mathcal{D}$ is non-empty because $(a_1,\cdots,a_r)\in\mathcal{D}$. Moreover, positive definiteness is an open condition on the coefficients of a quadratic form, so $\mathcal{D}$ is open.

Let $\mathcal{S}:=\{\lambda\in\R^r:\mathcal{Z}(\lambda)<+\infty\}$. If $\lambda\in\mathcal{D}$, then there exists $C>0$ such that $Q_{H,\lambda}(x)\geq C\|x\|^2$
for all $x\in V=\Lambda\otimes_{\Z}\R$. Therefore
\[\mathcal{Z}(\lambda)=\sum_{x\in\Lambda}e^{-Q_{H,\lambda}(x)}\leq\sum_{x\in\Lambda}e^{-C\|x\|^2}<+\infty.\]
So $\mathcal{D}\subseteq\mathcal{S}$. Since $\mathcal{D}$ is open, we have $\mathcal{D}\subseteq\mathcal{S}^{\circ}$.

Conversely, suppose that $\lambda\notin\mathcal{D}$. Then there exists $v\in V\setminus\{0\}$ such that $Q_{H,\lambda}(v)\leq0$. For every $\varepsilon>0$, consider $\lambda_{\varepsilon}:=\lambda-\varepsilon(a_1,\cdots,a_r)$. Since $Q_{H,(a_1,\cdots,a_r)}=H$ and $H(v)>0$, we have
\[Q_{H,\lambda_{\varepsilon}}(v)=Q_{H,\lambda}(v)-\varepsilon H(v)<0.\]
Since $\Lambda\otimes_{\Z}\Q$ is dense in $V$, there exists a nonzero $x\in\Lambda$ such that $Q_{H,\lambda_{\varepsilon}}(x)<0$. Hence,
\[\mathcal{Z}(\lambda_{\varepsilon})\geq\sum_{n\geq1}e^{-Q_{H,\lambda_{\varepsilon}}(nx)}=\sum_{n\geq1}e^{-n^2Q_{H,\lambda_{\varepsilon}}(x)}=+\infty,\]
which means $\lambda_{\varepsilon}\notin\mathcal{S}$ for every $\varepsilon>0$. Since $\lambda_{\varepsilon}\to\lambda$ as $\varepsilon\downarrow0$, it follows that $\lambda\notin\mathcal{S}^{\circ}$. Therefore $\mathcal{S}^{\circ}\subseteq\mathcal{D}$. Consequently, $\mathcal{D}=\mathcal{S}^{\circ}$.
\end{proof}

Similar to \S\ref{sec:one-model}, for $\lambda\in\mathcal{D}$, we define
\[\bm{\psi}(\lambda):=\log\mathcal{Z}(\lambda),\quad\bm{p}_{\lambda}(x):=\mathcal{Z}(\lambda)^{-1}e^{-Q_{H,\lambda}(x)}\]
and
\[\bm{m}(\lambda):=\E_{\bm{p}_{\lambda}}[(F_1,\cdots,F_r)],\quad\bm{\sigma}^2(\lambda):=\mathrm{Cov}_{\bm{p}_{\lambda}}[(F_1,\cdots,F_r)].\]
Note that $p_{\beta}=\bm{p}_{\beta(a_1,\cdots,a_r)}$ if $\beta(a_1,\cdots,a_r)\in\mathcal{D}$. Like Proposition~\ref{prop:thermodynamic-functions}, it is easy to show $\mathcal{Z}$, $\bm{\psi}$ are real analytic on $\mathcal{D}$, and $\nabla\bm{\psi}(\lambda)=-\bm{m}(\lambda)$, $\nabla^2\bm{\psi}(\lambda)=\bm{\sigma}^2(\lambda)$.

Likewise, the canonical ensemble as in \S\ref{sec:canonical-ensemble} can be obtained through copies of a single particle. Regarding the (exact-shell) microcanonical ensemble, we present the following high-dimensional generalization of Definition~\ref{def:exact-shell}.

\begin{lemma}\label{lem:equations}
Let $N\geq1$, suppose $n=(n_1,\cdots,n_r)\in\Z^r$ and write $x=(x_1,\cdots,x_N)\in\Lambda^N$, we have
\[\left\{x\in\Lambda^N:H_N(x)=\sum_{j=1}^ra_jn_j\right\}=\left\{x\in\Lambda^N:\sum_{i=1}^NF_j(x_i)=n_j,1\leq j\leq r\right\}.\]
\end{lemma}
\begin{proof}
The elements $a_1,\cdots,a_r$ are linearly independent over $\Q$, so $H_N(x)=\sum_{j=1}^ra_jn_j$ is equivalent to
\[\sum_{j=1}^ra_j\left(\sum_{i=1}^NF_j(x_i)-n_j\right)=0,\]
and hence to the $r$ equalities $\sum_{i=1}^NF_j(x_i)=n_j$.
\end{proof}

Hereafter, we always assume that the set in Lemma~\ref{lem:equations} is non-empty, which corresponds to the reachable condition in Definition~\ref{def:exact-shell}.

In order to state the first theorem of this section, we introduce the notion of total variation. For fixed $k$ and $N\geq k$, let $\pi_k:\Lambda^N\to\Lambda^k$ denote the projection onto the first $k$ coordinates. For probability measures on a countable set, we define
\[\|\mu-\nu\|_{\TV}:=\sup_A|\mu(A)-\nu(A)|=\frac{1}{2}\sum_x|\mu(x)-\nu(x)|.\]

\begin{theorem}\label{thm:h-equivalent}
Let $\lambda\in\mathcal{D}$. Assume that $\langle(F_1(x),\cdots,F_r(x)):x\in\Lambda\rangle_{\Z}=\Z^r$. For $N\geq1$, suppose there are $\lambda_N\in\mathcal{D}$ such that $\lambda_N\to\lambda$, $N\to+\infty$. Let $n_N=(n_{N,1},\cdots,n_{N,r})\in\Z^r$ be reachable such that
\[\Xi_N:=n_N-N\bm{m}(\lambda_N)=O(1),\quad N\to+\infty.\]
Consider the exact-shell microcanonical ensemble $\mu_{N,E_N}^{\sh}$ on $\Gamma_N(E_N)$, where $E_N=\sum_{j=1}^ra_jn_{N,j}$. For every $k\geq1$,
\begin{equation}\label{eq:tv-convergence}
\left\|(\pi_k)_*\mu^{\sh}_{N,E_N}-\bm{p}_{\lambda}^{\otimes k}\right\|_{\TV}\to0,\quad N\to+\infty.
\end{equation}
\end{theorem}

The proof of this theorem requires a high-dimensional parameter-uniform local central limit theorem for lattice-valued random variables. For convenience, and because this result differs from the conventional local central limit theorem with a fixed distribution, we restate it along with its proof.

\begin{proposition}[Uniform $r$-LCLT]\label{prop:h-lclt}
Assume $\langle(F_1(x),\cdots,F_r(x)):x\in\Lambda\rangle_{\Z}=\Z^r$. For $\lambda\in\mathcal{D}$, let $X_{1,\lambda},X_{2,\lambda},\cdots$ be independent copies of a random variable with law $\bm{p}_{\lambda}$, and let
\[S_{N,\lambda}:=\sum_{i=1}^N(F_1(X_{i,\lambda}),\cdots,F_r(X_{i,\lambda}))\]
be a random vector. Then, for every compact set $K\Subset\mathcal{D}$,
\[\sup_{\lambda\in K}\sup_{n\in\Z^r}\left|N^{\frac{r}{2}}\Prob_{\bm{p}_{\lambda}}(S_{N,\lambda}=n)-\frac{\exp\left(-\frac{1}{2}\Xi_{N,\lambda,n}\bm{\sigma}^2(\lambda)^{-1}\Xi^T_{N,\lambda,n}\right)}{(2\pi)^{\frac{r}{2}}\sqrt{\det\bm{\sigma}^2(\lambda)}}\right|\to0,\quad N\to+\infty,\]
where $\Xi_{N,\lambda,n}:=\frac{n-N\bm{m}(\lambda)}{\sqrt{N}}$.
\end{proposition}
\begin{proof}
First, we show that the covariance matrix is positive definite. Note that $t\bm{\sigma}^2(\lambda)t^T=\Var_{\bm{p}_{\lambda}}[t(F_1,\cdots,F_r)^T]\geq0$ for $t\in\R^r$. If $\Var_{\bm{p}_{\lambda}}[t(F_1,\cdots,F_r)^T]=0$, since $\bm{p}_{\lambda}$ has full support on $\Lambda$ and $F_j(0)=0$, we have $t(F_1,\cdots,F_r)^T\equiv0$. The assumption $\langle(F_1(x),\cdots,F_r(x)):x\in\Lambda\rangle_{\Z}=\Z^r$ then implies $t=0$. Thus $\bm{\sigma}^2(\lambda)$ is positive definite.

To describe the asymptotic behavior of the characteristic function, we show that the third absolute moments are uniformly bounded on $K$. Since $K\Subset\mathcal{D}$ and $\mathcal{D}$ is open, there exists $\eta>0$ such that
\[K_{\eta}:=\{\lambda-\eta\varepsilon:\lambda\in K,\varepsilon\in\{-1,1\}^r\}\Subset\mathcal{D}.\]
Let $\|\cdot\|_1$ denote the $1$-norm. We have
\begin{align*}
\E_{\bm{p}_{\lambda}}[e^{\eta\|(F_1,\cdots,F_r)\|_1}]&=\frac{1}{\mathcal{Z}(\lambda)}\sum_{x\in\Lambda}e^{-\lambda(F_1(x),\cdots,F_r(x))^T}e^{\eta\|(F_1(x),\cdots,F_r(x))\|_1}\\
&\leq\frac{1}{\mathcal{Z}(\lambda)}\sum_{\varepsilon\in\{-1,1\}^r}\sum_{x\in\Lambda}e^{-(\lambda-\eta\varepsilon)(F_1(x),\cdots,F_r(x))^T}\\
&=\frac{1}{\mathcal{Z}(\lambda)}\sum_{\varepsilon\in\{-1,1\}^r}\mathcal{Z}(\lambda-\eta\varepsilon).
\end{align*}
Note that $\mathcal{Z}\geq1$ and $\mathcal{Z}$ is continuous and bounded on the compact set $K_{\eta}$,
\[\sup_{\lambda\in K}\E_{\bm{p}_{\lambda}}[e^{\eta\|(F_1,\cdots,F_r)\|_1}]<+\infty.\]
It follows that, for every integer $i\geq1$, $\sup_{\lambda\in K}\E_{\bm{p}_{\lambda}}[\|(F_1,\cdots,F_r)\|^i]<+\infty$. Thus,
\[M_K:=\sup_{\lambda\in K}\E_{\bm{p}_{\lambda}}[\|(F_1,\cdots,F_r)-\bm{m}(\lambda)\|^3]<+\infty,\]
because $\|\bm{m}(\lambda)\|$ is uniformly bounded on $K$.

Consider the characteristic function $\phi_{\lambda}(t):=\E_{\bm{p}_{\lambda}}[e^{it(F_1,\cdots,F_r)^T}]$, $t\in\R^r$, and its centralization $\chi_{\lambda}(t):=e^{-it\bm{m}(\lambda)^T}\phi_{\lambda}(t)$. Now we have the following Taylor expansion
\[\chi_{\lambda}(t)=1-\frac{1}{2}t\bm{\sigma}^2(\lambda)t^T+R_{\lambda}(t),\]
where the remainder term $|R_{\lambda}(t)|\leq M_K\|t\|^3/6$. Hence, there exists $A_K>0$ such that
\[|\chi_{\lambda}(t)|^2\leq1-t\bm{\sigma}^2(\lambda)t^T+A_K\|t\|^3.\]
Note that there is $B_K>0$ such that $t\bm{\sigma}^2(\lambda)t^T\geq B_K\|t\|^2$ for every $\lambda\in K$ and every $t\in\R^r$, because $\lambda\mapsto\bm{\sigma}^2(\lambda)$ is continuous and $K$ is compact, this implies
\[|\chi_{\lambda}(t)|^2\leq1-B_K\|t\|^2+A_K\|t\|^3\leq1-\frac{B_K}{2}\|t\|^2\leq\exp\left(-\frac{B_K}{2}\|t\|^2\right)\]
for $\lambda\in K$ and $\|t\|\leq\frac{B_K}{2A_K}$. Therefore, $|\chi_{\lambda}(t)|\leq\exp(-\frac{B_K}{4}\|t\|^2)$, when $\|t\|\leq\delta_K:=\frac{B_K}{2A_K}$.

Next, we claim that $|\phi_{\lambda}(t)|<1$ for every $t\in[-\pi,\pi]^r\setminus\{0\}$ and every $\lambda\in\mathcal{D}$. Suppose
\[|\phi_{\lambda}(t)|=\left|\sum_{x\in\Lambda}\bm{p}_{\lambda}(x)e^{it(F_1(x),\cdots,F_r(x))^T}\right|=1.\]
Since this summation is a convex combination, with strictly positive coefficients, of complex numbers of modulus one, equality in the triangle inequality implies that all the $e^{it(F_1(x),\cdots,F_r(x))^T}$ are equal. Since $F_j(0)=0$, their common value is $1$. Thus, for $x\in\Lambda$,
\[t(F_1(x),\cdots,F_r(x))^T\in2\pi\Z.\]
Because $\langle(F_1(x),\cdots,F_r(x)):x\in\Lambda\rangle_{\Z}=\Z^r$, we obtain $t\in2\pi\Z^r$. Since $t\in[-\pi,\pi]^r$, this forces $t=0$, a contradiction. So far, we have $|\phi_{\lambda}(t)|<1$ on the compact set
\[W_K:=\{t\in[-\pi,\pi]^r:\|t\|\geq\delta_K\}.\]
Note that $(\lambda,t)\mapsto\phi_{\lambda}(t)$ is continuous, there exists $\rho_K<1$ such that $\sup_{\lambda\in K}\sup_{t\in W_K}|\phi_{\lambda}(t)|\leq\rho_K$. Since $|\chi_{\lambda}(t)|=|\phi_{\lambda}(t)|$, the same estimate holds for $\chi_{\lambda}$.

We then show
\begin{equation}\label{eq:sup-int}
\sup_{\lambda\in K}\int_{\R^r}|f_{N,\lambda}(t)-g_{\lambda}(t)|dt\to0,\quad N\to+\infty,
\end{equation}
where
\[f_{N,\lambda}(t):=\one_{[-\pi\sqrt{N},\pi\sqrt{N}]^r}(t)\chi_{\lambda}\left(\frac{t}{\sqrt{N}}\right)^N\quad\text{and}\quad g_{\lambda}(t):=\exp\left(-\frac{1}{2}t\bm{\sigma}^2(\lambda)t^T\right).\]
Fix a suitable $C>0$. If $\|t\|\leq C$, the uniform Taylor expansion gives
\[\chi_{\lambda}\left(\frac{t}{\sqrt{N}}\right)=1-\frac{1}{2N}t\bm{\sigma}^2(\lambda)t^T+O(N^{-3/2}),\quad N\to+\infty,\]
therefore,
\[N\log\chi_{\lambda}\left(\frac{t}{\sqrt{N}}\right)=-\frac{1}{2}t\bm{\sigma}^2(\lambda)t^T+O(N^{-1/2}),\quad N\to+\infty,\]
which means $\chi_{\lambda}(\frac{t}{\sqrt{N}})^N\to\exp(-\frac{1}{2}t\bm{\sigma}^2(\lambda)t^T)$ uniformly for $\lambda\in K$ and $\|t\|\leq C$. So
\begin{equation}\label{eq:sup-int-1}
\sup_{\lambda\in K}\int_{\|t\|\leq C}|f_{N,\lambda}(t)-g_{\lambda}(t)|dt\to0,\quad N\to+\infty.
\end{equation}
If $t\in[-\pi\sqrt{N},\pi\sqrt{N}]^r$ such that $\|t\|>\delta_K\sqrt{N}$, we have $|\chi_{\lambda}(\frac{t}{\sqrt{N}})|^N\leq\rho_K^N$, then
\[\sup_{\lambda\in K}\int_{t\in[-\pi\sqrt{N},\pi\sqrt{N}]^r,\|t\|>\delta_K\sqrt N}|f_{N,\lambda}(t)|dt\leq(2\pi\sqrt{N})^r\rho_K^N\to0,\quad N\to+\infty.\]
Recall that $g_{\lambda}(t)\leq\exp(-\frac{B_K}{2}\|t\|^2)$, as the integration domains decrease along a neighborhood basis at infinity, the integrals of $g_{\lambda}$ converge to $0$ uniformly in $\lambda\in K$. Since outside $[-\pi\sqrt{N},\pi\sqrt{N}]^r$ one has $f_{N,\lambda}\equiv0$, we obtain
\begin{equation}\label{eq:sup-int-2}
\sup_{\lambda\in K}\int_{\|t\|>\delta_K\sqrt{N}}|f_{N,\lambda}(t)-g_{\lambda}(t)|dt\to0,\quad N\to+\infty.
\end{equation}
For $C<\|t\|\leq\delta_K\sqrt{N}$, we already know
\[\left|\chi_{\lambda}\left(\frac{t}{\sqrt{N}}\right)\right|^N\leq\exp\left(-\frac{B_KN}{4}\left\|\frac{t}{\sqrt{N}}\right\|^2\right)=\exp\left(-\frac{B_K}{4}\|t\|^2\right),\]
so
\[\sup_{\lambda\in K}\int_{C<\|t\|\leq\delta\sqrt{N}}|f_{N,\lambda}(t)-g_{\lambda}(t)|dt\leq\int_{\|t\|>C}\left(\exp\left(-\frac{B_K}{4}\|t\|^2\right)+\exp\left(-\frac{B_K}{2}\|t\|^2\right)\right)dt.\] The right-hand side tends to $0$ as $C\to+\infty$. Combining this with \eqref{eq:sup-int-1} and \eqref{eq:sup-int-2}, and first letting $N\to+\infty$ and then $C\to+\infty$, proves \eqref{eq:sup-int}.

Finally, since $S_{N,\lambda}$ is $\Z^r$-valued, by the Fourier inversion formula,
\[\Prob_{\bm{p}_{\lambda}}(S_{N,\lambda}=n)=\frac{1}{(2\pi)^r}\int_{t\in[-\pi,\pi]^r}\phi_{\lambda}(t)^Ne^{-itn^T}dt=\frac{1}{(2\pi)^r}\int_{t\in[-\pi,\pi]^r}\chi_{\lambda}(t)^Ne^{-it(n-N\bm{m}(\lambda))^T}dt.\]
After the change of variables $s=\sqrt{N}t$, we obtain
\[N^{r/2}\Prob_{\bm{p}_{\lambda}}(S_{N,\lambda}=n)=\frac{1}{(2\pi)^r}\int_{\R^r}f_{N,\lambda}(s)e^{-is\Xi_{N,\lambda,n}^T}ds.\]
By \eqref{eq:sup-int}, when $N\to+\infty$,
\[\sup_{\lambda\in K}\sup_{n\in\Z^r}
\left|N^{r/2}\Prob_{\bm{p}_{\lambda}}(S_{N,\lambda}=n)-\frac{1}{(2\pi)^r}\int_{\R^r}g_{\lambda}(s)e^{-is\Xi_{N,\lambda,n}^T}ds\right|
\leq\frac{1}{(2\pi)^r}\sup_{\lambda\in K}\|f_{N,\lambda}-g_{\lambda}\|_{L^1(\R^r)}\to0.\]
It remains to compute the Fourier transform of the Gaussian kernel. Since $\bm{\sigma}^2(\lambda)$ is symmetric and positive definite, it has a symmetric positive-definite square root $\bm{\sigma}^2(\lambda)^{1/2}$. With the row-vector change of variables $u=s\bm{\sigma}^2(\lambda)^{1/2}$,
we obtain
\[\frac{1}{(2\pi)^r}\int_{\R^r}\exp\left(-\frac{1}{2}s\bm{\sigma}^2(\lambda)s^T\right)e^{-is\Xi_{N,\lambda,n}^T}ds
=\frac{1}{(2\pi)^r\sqrt{\det\bm{\sigma}^2(\lambda)}}\int_{\R^r}e^{-\frac{1}{2}uu^T}e^{-iu\bm{\sigma}^2(\lambda)^{-1/2}\Xi_{N,\lambda,n}^T}du.\]
The integral factors into $r$ one-dimensional Gaussian Fourier transforms. Therefore,
\[\frac{1}{(2\pi)^r}\int_{\R^r}\exp\left(-\frac{1}{2}s\bm{\sigma}^2(\lambda)s^T\right)e^{-is\Xi_{N,\lambda,n}^T}ds=
\frac{\exp\left(-\frac{1}{2}\Xi_{N,\lambda,n}\bm{\sigma}^2(\lambda)^{-1}\Xi_{N,\lambda,n}^T\right)}{(2\pi)^{\frac{r}{2}}\sqrt{\det\bm{\sigma}^2(\lambda)}}.\]
This completes the proof.
\end{proof}

\begin{proof}(of Theorem~\ref{thm:h-equivalent})
By Lemma~\ref{lem:equations}, we have
\[\Gamma_N(E_N)=\left\{x\in\Lambda^N:\sum_{i=1}^NF_j(x_i)=n_{N,j},1\leq j\leq r\right\}.\]
This shell corresponds to the event $\{S_{N,\lambda_N}=n_N\}$. By definition, on this event the product measure $\bm{p}_{\lambda_N}^{\otimes N}$ is constant. Hence,
\[\mu^{\sh}_{N,E_N}=\bm{p}_{\lambda_N}^{\otimes N}(\cdot|S_{N,\lambda_N}=n_N).\]
It follows that
\begin{equation}\label{eq:ratio}
(\pi_k)_*\mu^{\sh}_{N,E_N}(y)=\bm{p}_{\lambda_N}^{\otimes k}(y)\frac{\Prob_{\bm{p}_{\lambda_N}}(S_{N-k,\lambda_N}=n_N-A(y))}{\Prob_{\bm{p}_{\lambda_N}}(S_{N,\lambda_N}=n_N)},
\end{equation}
where $A(y):=\sum_{i=1}^k(F_1(y_i),\cdots,F_r(y_i))$, $y=(y_1,\cdots,y_k)\in\Lambda^k$. When $N\to+\infty$, $\lambda_N\to\lambda$, the denominator is centered because its displacement is $\Xi_N=O(1)$, while the displacement in the numerator is also
\[n_N-A(y)-(N-k)\bm{m}(\lambda_N)=\Xi_N+k\bm{m}(\lambda_N)-A(y)=O(1).\]
Note that the sequence $\{\lambda_N\}$ is compact, one can apply Proposition~\ref{prop:h-lclt} by considering $\Xi_{N,\lambda_N,n_N}$ and $\Xi_{N-k,\lambda_N,n_N-A}$. Therefore, it is easy to check the ratio in \eqref{eq:ratio} tends to $1$. Since
$\bm{p}_{\lambda_N}^{\otimes k}\to\bm{p}_{\lambda}^{\otimes k}$ as $N\to+\infty$, \eqref{eq:ratio} converges pointwise to $\bm{p}_{\lambda}^{\otimes k}$. By Scheff\'{e}'s lemma we obtain \eqref{eq:tv-convergence}.
\end{proof}

For $r\geq2$, an exact energy shell may therefore retain information that is invisible to the one-parameter canonical ensemble, see Example~\ref{ex:ex} for a concrete example. The special case $r=1$ is also required for our purposes. In this case, $H(\Lambda)$ is lattice and $G_H$ has $\Z$-rank $1$, so there is a unique maximal span $h>0$ characterized by
\[H(\Lambda)\subseteq h\Z,\quad\langle H(x)/h:x\in\Lambda\rangle_{\Z}=\Z.\]
The second condition is the aperiodicity of the normalized energy support. Now, for $\beta\in\mathcal{D}=(0,+\infty)$, the $X_{1,\beta},X_{2,\beta},\cdots$ above (in Proposition~\ref{prop:h-lclt}) then become independent copies of a random variable with law $p_{\beta}=\bm{p}_{\beta h}$. Let $S_{N,\beta}=\sum_{i=1}^NH(X_{i,\beta})$ be as before, we readily obtain:

\begin{corollary}[Uniform $1$-LCLT]\label{cor:1-lclt}
Let $K\Subset(0,+\infty)$ be a compact set and let $h$ be the maximal span of $H(\Lambda)$. Then
\[\sup_{\beta\in K}\sup_{E\in h\Z}\left|\sqrt{N}\Prob_{p_{\beta}}(S_{N,\beta}=E)-\frac{h}{\sqrt{2\pi\sigma^2(\beta)}}\exp\left(-\frac{(E-Nm(\beta))^2}{2N\sigma^2(\beta)}\right)\right|\to0,\quad N\to+\infty.\]
\end{corollary}

Under the lattice-valued assumption, we establish the canonical-microcanonical equivalence for the exact-shell microcanonical ensembles.

\begin{theorem}[Exact-shell Equivalence]\label{thm:lattice-equivalence}
Suppose $H(\Lambda)$ is lattice with maximal span $h>0$. Let $E_N\in\im(H_N)$ satisfy $E_N/N\to u>0$, and set $\beta=m^{-1}(u)$. Then for every $k\geq1$,
\begin{equation}\label{eq:lattice-TV}
\left\|(\pi_k)_*\mu^{\sh}_{N,E_N}-p_{\beta}^{\otimes k}\right\|_{\TV}\to0,\quad N\to+\infty.
\end{equation}
Moreover, let $u_N=E_N/N$ and $\beta_N=m^{-1}(u_N)$, then
\begin{equation}\label{eq:lattice-sharp-count}
g_N(E_N)=\frac{hZ(\beta_N)^Ne^{\beta_NE_N}}{\sqrt{2\pi N\sigma^2(\beta_N)}}(1+o(1)).
\end{equation}
Consequently, the thermodynamic limit
\begin{equation}\label{eq:lattice-entropy-limit}
\lim_{N\to+\infty}\frac{1}{N}R^{\sh}_{N,E_N}=\psi(\beta)+\beta u=s(u).
\end{equation}
\end{theorem}
\begin{proof}
Since $r=1$, \eqref{eq:lattice-TV} follows immediately from Theorem~\ref{thm:h-equivalent}. For every $E\in h\Z$,
\begin{equation}\label{eq:lattice-probability-count}
\Prob_{p_{\beta}}(S_{N,\beta}=E)=g_N(E)Z(\beta)^{-N}e^{-\beta E},
\end{equation}
where $g_N(E)=0$ when $E$ is not reachable. Since the sequence $\{\beta_N\}$ is compact, at the centered point $E_N=Nm(\beta_N)$ one has
\[\Prob_{p_{\beta_N}}(S_{N,\beta_N}=E_N)=\frac{h+o(1)}{\sqrt{2\pi N\sigma^2(\beta_N)}},\quad N\to+\infty,\]
by Corollary~\ref{cor:1-lclt}. Solving \eqref{eq:lattice-probability-count} for $g_N(E_N)$ proves \eqref{eq:lattice-sharp-count}. Taking logarithms, dividing by $N$, and using $\beta_N\to\beta$ proves \eqref{eq:lattice-entropy-limit}.
\end{proof}

\begin{remark}
We make the following two remarks:
\begin{itemize}
\item In principle, the assumption $n_N-N\bm{m}(\lambda_N)=O(1)$ in Theorem~\ref{thm:h-equivalent} can be relaxed to $o(\sqrt{N})$. We choose the former not only to interpret the internal energy (as becomes more evident in Theorem~\ref{thm:lattice-equivalence}) as the average kinetic energy, but also for another key advantage: it directly covers the construction of rounding continuous expectations to their nearest integers.

\item The reachability assumption $E_N\in\im(H_N)$ in Theorem~\ref{thm:lattice-equivalence} is part of the statement, not a technical convenience.  In contrast to Theorem~\ref{thm:equivalence} below, an exact arithmetic shell can be empty along infinitely many values of $N$ unless the energy sequence respects the additive semigroup generated by $H(\Lambda)$.
\end{itemize}
\end{remark}

\section{Equivalence of Windowed Ensembles}\label{sec:window}
In this section, we investigate the thermodynamic limit behavior of the windowed microcanonical ensemble. By Theorem~\ref{thm:h-equivalent}, if $H(\Lambda)$ is nonlattice, the exact-shell microcanonical ensemble $\mu_{N,E_N}^{\sh}$ may not equivalent to the canonical ensemble $\mu_{N,\beta}^{\can}$. However, the windowed version discussed herein recovers the requisite ensemble equivalence. It should be noted that even without treating the lattice case of $H(\Lambda)$ separately, the approach developed in this section can be applied to it, provided that the window is chosen to be the maximal span.

We also consider independent copies $X_{1,\beta},X_{2,\beta},\cdots$ of a random variable with law $p_{\beta}$, write $S_{N,\beta}=\sum_{i=1}^NH(X_{i,\beta})$. Define the characteristic function $\varphi_{\beta}(t):=\E_{p_{\beta}}[e^{itH}]$ and let $\eta_{\beta}(t):=e^{-itm(\beta)}\varphi_{\beta}(t)$ be the centered characteristic function.

We present a windowed local central limit theorem, which is needed because when $H(\Lambda)$ is nonlattice, the appropriate local object is a bounded energy window rather than a point shell. The fixed-distribution version of this theorem was originally proved by Stone \cite{Stone} (or see \cite[\S1]{Borovkov} for a clearer description), and, analogously to \S\ref{sec:lattice-exact}, we require a uniform version here. We include a brief proof, which follows a line of argument largely identical to the proof of Proposition~\ref{prop:h-lclt}.

\begin{proposition}[Uniform Stone-LCLT]\label{thm:uniform-stone}
Suppose that $H(\Lambda)$ is nonlattice. Let $K\Subset(0,+\infty)$ be a compact set and let $\Delta>0$. Then
\[\sup_{\beta\in K}\sup_{x\in\R}\left|\sqrt{N}\Prob_{p_{\beta}}(S_{N,\beta}\in[x,x+\Delta))-\frac{\Delta}{\sqrt{2\pi\sigma^2(\beta)}}\exp\left(-\frac{(x-Nm(\beta))^2}{2N\sigma^2(\beta)}\right)\right|\to0,\quad N\to+\infty.\]
\end{proposition}
According to the same arguments as in the proof of Proposition~\ref{prop:h-lclt}, we obtain:
\begin{itemize}
\item $\sup_{\beta\in K}\E_{p_{\beta}}[e^{\eta H}]<+\infty$ for some $\eta>0$, so the third absolute moments are uniformly bounded on $K$.

\item For every $0<\delta<T<\infty$, there is $\rho<1$, depending on $K,\delta,T$, such that
\[\sup_{\beta\in K}\sup_{\delta\leq|t|\leq T}|\varphi_{\beta}(t)|\leq\rho.\]
Indeed, if $|\varphi_{\beta}(t)|=1$ for some $t\neq0$, then one can conclude $H(\Lambda)$ is lattice, a contradiction.
\end{itemize}

We also need the following fact:
\begin{lemma}[Beurling-Selberg, {\cite[\S1]{Vaaler}}]\label{lem:fourier-sandwich}
For $a<b$, define the symmetrized interval characteristic function by
\[\chi_{a,b}(y)=\one_{(a,b)}(y)+\frac{1}{2}\one_{\{a\}}(y)+\frac{1}{2}\one_{\{b\}}(y).\]
For every $\varepsilon>0$, there exist bounded, real-valued, integrable functions $g^-_{\varepsilon}$ and $g^+_{\varepsilon}$ such that
\[g^-_{\varepsilon}\leq\chi_{a,b}\leq g^+_{\varepsilon}\quad\text{and}\quad\int_{\R}(g^+_{\varepsilon}-g^-_{\varepsilon})dy<\varepsilon,\]
and both Fourier transforms are continuous and compactly supported.
\end{lemma}

\begin{proof}(of Proposition~\ref{thm:uniform-stone})
Since the centered third absolute moments are uniformly bounded on $K$, we have
\[\eta_{\beta}(t)=1-\frac{\sigma^2(\beta)t^2}{2}+O(|t|^3),\quad t\to0,\]
uniformly in $\beta\in K$. The compactness of $K$ then implies that for some $c,\delta_0>0$,
\[|\eta_\beta(t)|\leq e^{-ct^2},\quad\beta\in K,|t|\leq\delta_0.\]
Moreover, since $H(\Lambda)$ is nonlattice, for every $T>\delta_0$ there exists $\rho<1$, such that
\[\sup_{\beta\in K}\sup_{\delta_0\leq|t|\leq T}|\eta_{\beta}(t)|\leq\rho.\]

We first record a smoothing estimate. Let $g$ be a function that is the continuous inverse Fourier transform of a compactly supported integrable function $\widehat{g}$, say $\mathrm{supp}(\widehat{g})\subseteq[-T,T]$. By the Fourier inversion formula, for $y\in\R$,
\[\sqrt{N}\E_{p_{\beta}}[g(S_{N,\beta}-Nm(\beta)-y)]=\frac{1}{2\pi}\int_{\R}\widehat{g}\left(\frac{t}{\sqrt{N}}\right)\eta_{\beta}\left(\frac{t}{\sqrt{N}}\right)^Ne^{-ity/\sqrt{N}}dt.\]
The preceding estimates imply
\[\sup_{\beta\in K}\int_{\R}\left|\widehat{g}\left(\frac{t}{\sqrt{N}}\right)\eta_{\beta}\left(\frac{t}{\sqrt{N}}\right)^N-\widehat{g}(0)e^{-\sigma^2(\beta)t^2/2}\right|dt\to0,\quad N\to+\infty.\]
Indeed, on bounded $t$-intervals this follows from the uniform Taylor expansion; on $|t|\leq\delta_0\sqrt{N}$ the integrand is dominated by a Gaussian, while on $\delta_0\sqrt{N}\leq|t|\leq T\sqrt{N}$ it is $O(\rho^N)$. Since $\widehat{g}(0)=\int_{\R}g$, the Gaussian Fourier transform yields, uniformly in $\beta\in K$ and $y\in\R$,
\[\sqrt{N}\E_{p_{\beta}}[g(S_{N,\beta}-Nm(\beta)-y)]=\frac{\int_{\R}g(u)du}{\sqrt{2\pi\sigma^2(\beta)}}\exp\left(-\frac{y^2}{2N\sigma^2(\beta)}\right)+o(1),\quad N\to+\infty.\]

We now apply this estimate to the interval $[0,\Delta)$. Let $0<\delta<\frac{\Delta}{2}$ and $\varepsilon>0$. By Lemma~\ref{lem:fourier-sandwich}, there exist bounded integrable functions $g^-_{\delta,\varepsilon}$ and $g^+_{\delta,\varepsilon}$, with continuous compactly supported Fourier
transforms, such that
\[g^-_{\delta,\varepsilon}\leq\chi_{\delta,\Delta-\delta}\leq\one_{[0,\Delta)}\leq\chi_{-\delta,\Delta+\delta}\leq g^+_{\delta,\varepsilon},\]
and
\[\Delta-2\delta-\varepsilon\leq\int_{\R}g^-_{\delta,\varepsilon}(u)du\leq\Delta;\quad\Delta\leq\int_{\R}g^+_{\delta,\varepsilon}(u)du\leq\Delta+2\delta+\varepsilon.\]
Taking $y=x-Nm(\beta)$ and applying the smoothing estimate to $g^-_{\delta,\varepsilon}$ and $g^+_{\delta,\varepsilon}$, we obtain
\[\limsup_{N\to+\infty}\sup_{\beta\in K}\sup_{x\in\R}\left|\sqrt{N}\Prob_{p_{\beta}}(S_{N,\beta}\in[x,x+\Delta))-\frac{\Delta}{\sqrt{2\pi\sigma^2(\beta)}}\exp\left(-\frac{(x-Nm(\beta))^2}{2N\sigma^2(\beta)}\right)\right|\leq C(2\delta+\varepsilon),\]
where $C$ depends only on $K$. Letting first $\varepsilon\downarrow0$ and then $\delta\downarrow0$ proves the result.
\end{proof}

Let $K\Subset(0,+\infty)$ be a compact set and $\beta\in K$. For $x\in\R$ and $\Delta>0$, denote
\[J_{N,\beta}(x,\Delta):=\E_{p_{\beta}}[e^{\beta(S_{N,\beta}-x)}\one_{\{0\leq S_{N,\beta}-x<\Delta\}}].\]
By Proposition~\ref{thm:uniform-stone}, we obtain
\begin{equation}\label{eq:exp-uniform-stone}
J_{N,\beta}(x,\Delta)=\frac{\int_0^{\Delta}e^{\beta t}dt}{\sqrt{2\pi N\sigma^2(\beta)}}\exp\left(-\frac{(x-Nm(\beta))^2}{2N\sigma^2(\beta)}\right)+o(N^{-1/2}),\quad N\to+\infty,
\end{equation}
uniformly for $\beta\in K$ and $x\in\R$. This asymptotic formula is a locally weighted version of Proposition~\ref{thm:uniform-stone}. Based on this, the counting identity, like \eqref{eq:lattice-probability-count}, is
\begin{equation}\label{eq:counting-identity}
G_N(E,\Delta)=Z(\beta)^Ne^{\beta E}J_{N,\beta}(E,\Delta).
\end{equation}
Indeed, summing $1=Z(\beta)^Ne^{\beta H_N(x)}\mu^{\can}_{N,\beta}(x)$ over the window yields \eqref{eq:counting-identity}.

\begin{theorem}[Windowed Equivalence]\label{thm:equivalence}
Suppose $H(\Lambda)$ is nonlattice. Fix $\Delta>0$, let $E_N\in\R$ satisfy $E_N/N\to u>0$, and set $\beta=m^{-1}(u)$. Then $G_N(E_N,\Delta)>0$ for all sufficiently large $N$, and for every fixed $k\geq1$,
\begin{equation}\label{eq:TV-equivalence}
\left\|(\pi_k)_*\mu^{\mc}_{N,E_N,\Delta}-p_\beta^{\otimes k}\right\|_{\TV}\to0,\quad N\to+\infty.
\end{equation}
Moreover, let $u_N=E_N/N$ and $\beta_N=m^{-1}(u_N)$, then
\begin{equation}\label{eq:sharp-count}
G_N(E_N,\Delta)=\frac{Z(\beta_N)^Ne^{\beta_NE_N}}{\sqrt{2\pi N\sigma^2(\beta_N)}}\left(\int_0^{\Delta}e^{\beta_Nt}dt\right)(1+o(1)).
\end{equation}
Consequently, the thermodynamic limit
\begin{equation}\label{eq:entropy-limit}
\lim_{N\to+\infty}\frac{1}{N}R_{N,E_N,\Delta}^{\mc}=\psi(\beta)+\beta u=s(u).
\end{equation}
\end{theorem}
\begin{proof}
Write $u_N=E_N/N$ and $\beta_N=m^{-1}(u_N)$. Since $\beta_N\to\beta$, the sequence $\{\beta_N\}\subseteq(0,+\infty)$ is compact. By \eqref{eq:exp-uniform-stone},
\[J_{N,\beta_N}(E_N,\Delta)=\frac{\int_0^{\Delta}e^{\beta_Nt}dt}{\sqrt{2\pi N\sigma^2(\beta_N)}}+o(N^{-1/2})>0\]
for all sufficiently large $N$. Hence $G_N(E_N,\Delta)>0$ for sufficiently large $N$.

Let $y=(y_1,\cdots,y_k)\in\Lambda^k$ and write $A(y):=\sum_{i=1}^kH(y_i)$. By \eqref{eq:counting-identity},
\[(\pi_k)_*\mu^{\mc}_{N,E_N,\Delta}(y)=\frac{G_{N-k}(E_N-A(y),\Delta)}{G_N(E_N,\Delta)}=p_{\beta_N}^{\otimes k}(y)\frac{J_{N-k,\beta_N}(E_N-A(y),\Delta)}{J_{N,\beta_N}(E_N,\Delta)}.\]
The denominator is centered because $E_N=Nm(\beta_N)$, and the numerator has displacement $(E_N-A)-(N-k)m(\beta_N)=ku_N-A=O(1)$. Hence, by \eqref{eq:exp-uniform-stone} again, we have
\[J_{N,\beta_N}(E_N,\Delta)=\frac{c_N+o(1)}{\sqrt{N}}\quad\text{and}\quad J_{N-k,\beta_N}(E_N-A(y),\Delta)=\frac{c_N+o(1)}{\sqrt{N-k}},\]
where
\[c_N=\frac{\int_0^{\Delta}e^{\beta_Nt}dt}{\sqrt{2\pi\sigma^2(\beta_N)}}\to\frac{\int_0^{\Delta}e^{\beta t}dt}{\sqrt{2\pi\sigma^2(\beta)}}>0,\quad N\to+\infty.\]
The ratio of the two $J$-terms tends to $1$, while $p_{\beta_N}^{\otimes k}(y)\to p_{\beta}^{\otimes k}(y)$. So $(\pi_k)_*\mu^{\mc}_{N,E_N,\Delta}\to p_{\beta}^{\otimes k}$ converges pointwise on $\Lambda^k$. Now \eqref{eq:TV-equivalence} follows from the Scheff\'{e}'s lemma.

Formula \eqref{eq:sharp-count} is \eqref{eq:counting-identity} combined with \eqref{eq:exp-uniform-stone} at the centered point $E_N=Nm(\beta_N)$. Taking logarithms we have
\[\frac{1}{N}\log G_N(E_N,\Delta)=\psi(\beta_N)+\beta_Nu_N-\frac{\log N}{2N}+O(1/N),\quad N\to+\infty,\]
where the $O(1/N)$ term is uniform along the sequence because the remaining prefactor converges to a positive finite limit. Since $u_N\to u$ and $\beta_N\to\beta$, this proves \eqref{eq:entropy-limit}.
\end{proof}

The window width $\Delta$ is fixed in Theorem~\ref{thm:equivalence}. The limiting marginal is independent of $\Delta$, although the subexponential prefactor in the state count depends on it.

\section{Comparison with Bost's Thermodynamic Formalism}\label{sec:bost}
Let $F$ be a number field and let $\overline{L}$ be a Hermitian line bundle over $\Spec(\mathcal{O}_F)$: its finite part is a projective rank one $\mathcal{O}_F$-module $L$, and its archimedean part consists of conjugation-invariant Hermitian metrics. The underlying abelian group of $L$ is free of rank $[F:\mathbb Q]$. If $a_{\tau}$ denotes the image of $a\in L$ in the one-dimensional archimedean fibre at $\tau$, the metrics define a Euclidean norm
\[\|a\|^2:=\sum_{\tau:F\hookrightarrow\R}\|a_{\tau}\|_{\tau}^2+2\sum_{\{\tau:F\hookrightarrow\mathbb{C}\}/\mathrm{conjugation}}\|a_{\tau}\|_{\tau}^2.\]
Thus $(L,\|\cdot\|)$ is a Euclidean lattice.

Now, the construction in \S\ref{sec:lattice-exact} and \S\ref{sec:window} applies. Whether $H(L)$ is lattice or nonlattice is a property of the metric values, not merely of the field $F$. If $H(L)$ is lattice, Theorem~\ref{thm:lattice-equivalence} applies to every reachable exact-energy sequence. If $H(L)$ is nonlattice, Theorem~\ref{thm:equivalence} applies to
\[\#\left\{(a_1,\cdots,a_N)\in L^N:E_N\leq\sum_{i=1}^N\|a_i\|^2<E_N+\Delta\right\}.\]

\begin{example}\label{ex:ex}
Consider $F=\Q(\sqrt2)$ and $L=\mathcal{O}_F=\Z[\sqrt2]$.  Define positive archimedean weights so that, for $a,b\in\Z$,
\[H(a+b\sqrt2)=(a+b\sqrt2)^2+2(a-b\sqrt2)^2=3(a^2+2b^2)-2\sqrt2ab.\]
It is clear that the energy set $H(L)$ is nonlattice. The windowed theory therefore applies, whereas exact shells may encode the two integer statistics $\sum(a_i^2+2b_i^2)$ and $\sum a_ib_i$ separately. Adopting the notations in \S\ref{sec:lattice-exact}, one has
\[r=2,\quad (a_1,a_2)=(3,-2\sqrt2),\quad (F_1,F_2)=(a^2+2b^2,ab).\]
The generalized Gibbs law associated with $\lambda=(1,0)$ is proportional to $e^{-(a^2+2b^2)}$ by Theorem~\ref{thm:h-equivalent}, which is not proportional to $e^{-\beta H}$.
\end{example}

The logarithmic theta series controls both the local Gibbs limit and the exponential growth rate of lattice points in high-dimensional shells. The thermodynamic entropy density function $s(u)$ records arithmetic information of Hermitian line bundles: the Arakelov degree is recovered from the renormalized high-temperature partition function by the Poisson summation formula \cite[\S3.1]{Bost},
\[\log Z(\beta)=\frac{d}{2}\log\frac{\pi}{\beta}+\widehat{\deg}_{\mathcal{O}_F}(\overline{L})-\frac{1}{2}\log|D_F|+o(1),\quad\beta\downarrow0,\]
where $d=[F:\Q]$ and $\widehat{\deg}_{\mathcal{O}_F}(\overline{L})=\frac{1}{2}\log|D_F|-\log\mathrm{covol}(\overline{L})$. So we have
\begin{equation}\label{eq:su}
s(u)=\log Z(\beta)+\beta u=\frac{d}{2}\left(1+\log\frac{2\pi u}{d}\right)+\left(\widehat{\deg}_{\mathcal{O}_F}(\overline{L})-\frac{1}{2}\log|D_F|\right)+o(1),\quad u\to+\infty,
\end{equation}
since one can verify there is $c>0$ such that
\[u=m(\beta)=-\frac{d}{d\beta}\log Z(\beta)=\frac{d}{2\beta}+O(\beta^{-2}e^{-c/\beta}),\quad\beta\downarrow0.\]

The physical meaning of the first term in \eqref{eq:su} is the discrete version of the thermodynamic entropy in the classical sense, while the second term serves as a correction provided by the arithmetic information of the lattice.

Bost's framework \cite[\S5-\S6]{Bost} starts from a measure space $(\mathcal{E},\mathcal{T},\mu)$ and a nonnegative Hamiltonian $H_{\mathrm{B}}$. It associates
\[Z_{\mathrm{B}}(t):=\int_{\mathcal{E}}e^{-tH_{\mathrm{B}}}d\mu,\quad\Psi_{\mathrm{B}}(t):=\log Z_{\mathrm{B}}(t),\]
and the cumulative $N$-fold sublevel measure
\[A_N(E):=\mu^{\otimes N}\{H_{\mathrm{B},N}\leq NE\}.\]
Under his basic integrability assumptions, Bost proved that
\begin{equation}\label{eq:bost-main}
S_{\mathrm{B}}(E):=\lim_{N\to+\infty}\frac{1}{N}\log A_N(E)=\inf_{t>0}\{\Psi_{\mathrm{B}}(t)+tE\},
\end{equation}
with the corresponding thermodynamic identities.

For a Euclidean lattice, Bost chose counting measure and $H_{\mathrm{B}}(x)=\pi\|x\|^2$. With the normalization of the present paper,
\begin{equation}\label{eq:bost-normalization}
\Psi_{\mathrm{B}}(t)=\psi(\pi t),\quad A_N(\pi u)=\#\{x\in\Lambda^N:H_N(x)\leq Nu\}.
\end{equation}
Changing variables $\beta=\pi t$ in \eqref{eq:bost-main} yields $S_{\mathrm{B}}(\pi u)=s(u)$. In Bost's notation this is the relation $\widetilde{h}^0_{\mathrm{Ar}}(\overline{L},u)=s(u)$, up to precisely the normalization displayed in \eqref{eq:bost-normalization}.

The overlap is therefore complete at exponential scale: Bost's cumulative ball count, the lattice count in Theorem~\ref{thm:lattice-equivalence}, and the fixed-width nonlattice count in Theorem~\ref{thm:equivalence} all have logarithmic growth rate $s(u)$, this is actually the entropy. Furthermore, \cite[Proposition~5.2.2]{Bost} permits shrinking intervals in the specific energy, but its stated hypothesis corresponds to a total-energy width of order at least $\sqrt{N}$. It therefore does not cover the fixed total width $\Delta$ used here.

Our results are local in a different sense. In the lattice regime it isolates one reachable coefficient by a lattice local central limit theorem; in the nonlattice regime it resolves an $O(1)$ total-energy window by Stone's theorem. Moreover, Theorem~\ref{thm:lattice-equivalence} and Theorem~\ref{thm:equivalence} identify the limiting law of every fixed subsystem, a statement not contained in the cumulative counting formula \eqref{eq:bost-main}. Thus Bost supplied the general thermodynamic and Legendre-Fenchel structure, while the present results add the arithmetic lattice/nonlattice dichotomy, local state-counting prefactors, and equivalence of ensembles at the level of finite marginals.

\enlargethispage{4\baselineskip}


\begin{thebibliography}{99}

\bibitem{Borovkov} A.A. Borovkov and K.A. Borovkov, \emph{A Refined Version of the Integro-local Stone Theorem}, Stat. Probab. Lett., {\bf 123} (2017), 153-159.
\bibitem{Bost} J.-B. Bost, \emph{Chapter IV: Euclidean Lattices, Theta Invariants, and Thermodynamic Formalism}, in E. Peyre and G. R\'{e}mond (eds.), Arakelov Geometry and Diophantine Applications, Lecture Notes in Mathematics, vol. 2276, Springer, 2021, 105-211.
\bibitem{Cancrini} N. Cancrini and S. Olla. Ensemble Dependence of Fluctuations: Canonical Microcanonical Equivalence of Ensembles. Journal of Statistical Physics, {\bf 168(4)} (2017), 707-730.
\bibitem{Diaconis} P. Diaconis and D. A. Freedman, \emph{Conditional Limit Theorems for Exponential Families and Finite Versions of de Finetti's Theorem}, J. Theoret. Probab., {\bf 1} (1988), 381-410.
\bibitem{Ellis} R.S. Ellis, \emph{Entropy, Large Deviations, and Statistical Mechanics}, Grundlehren der mathematischen Wissenschaften, vol. 271, Springer, 1985.
\bibitem{Lawler} G.F. Lawler and V. Limic, \emph{Random Walk: A Modern Introduction}, Cambridge University Press, 2012.
\bibitem{Plastino} A.R. Plastino and A. Plastino, \emph{From Gibbs Microcanonical Ensemble to Tsallis Generalized Canonical Distribution}, Physics Letters A, {\bf 193(2)} (1994), 140-143.
\bibitem{Shepp} L. A. Shepp, \emph{A Local Limit Theorem}, The Annals of Mathematical Statistics, {\bf 35(1)} (1964), 419-423.
\bibitem{Stone} C. Stone, \emph{A Local Limit Theorem for Nonlattice Multi-dimensional Distribution Functions}, Ann. Math. Statist., {\bf 36} (1965), 546-551.
\bibitem{Vaaler} J. D. Vaaler, \emph{Some Extremal Functions in Fourier Analysis}, Bull. Amer. Math. Soc., {\bf 12} (1985), 183-216.
\end{thebibliography}
\end{document}